\documentclass[12pt]{article}
\usepackage{amssymb, amsmath}
\newcommand{\qed}{\hfill \rule{2.5mm}{2.5mm}}

\newcommand{\cref}[1]{\cite{#1}}
\begin{document}
\newtheorem{thm}{Theorem}[section]
\newtheorem{defs}[thm]{Definition}
\newtheorem{lem}[thm]{Lemma}
\newtheorem{rem}[thm]{Remark}
\newtheorem{cor}[thm]{Corollary}
\newtheorem{prop}[thm]{Proposition}
\renewcommand{\theequation}{\arabic{section}.\arabic{equation}}
\newcommand{\newsection}[1]{\setcounter{equation}{0}\section{#1}}
\title{ Near-Epoch Dependence in Riesz Spaces
      }
      
\author{
{Wen-Chi Kuo}
 \footnote{Supported in part by NRF grant number CSUR160503163733.} \\
School of Mathematics, 
University of the Witwatersrand\\ 
Private Bag 3, P O WITS 2050, South Africa
\\ \\
{Michael J. Rogans}\\
School of Statistics and Actuarial Science,\\
University of the Witwatersrand\\
Private Bag 3, P O WITS 2050, South Africa\\ \\
{ Bruce A. Watson} \footnote{Supported in part by the Centre for Applicable Analysis and
Number Theory and by NRF grant number IFR170214222646 with grant no. 109289.}\\
School of Mathematics, 
University of the Witwatersrand\\
Private Bag 3, P O WITS 2050, South Africa}
\maketitle
\abstract{The abstraction of the study of stochastic processes to Banach lattices and vector lattices has received much attention
by Grobler, Kuo, Labuschagne, Stoica, Troitsky and Watson over the past fifteen years.
By contrast mixing processes have received very little attention. 
In particular mixingales were generalized to the Riesz space setting in
{\sc W.-C. Kuo, J.J. Vardy, B.A. Watson,}
Mixingales on Riesz spaces,
{\em J. Math. Anal. Appl.}, \textbf{402} (2013), 731-738.
The concepts of strong and uniform mixing as well as related mixing inequalities were extended to this setting in
{\sc W.-C. Kuo, M.J. Rogans, B.A. Watson,}
Mixing inequalities in Riesz spaces,
{\em J. Math. Anal. Appl.}, \textbf{456} (2017), 992-1004.
In the present work we formulate the concept of near-epoch dependence for Riesz space processes and show that if a process is near-epoch dependent and either strong or uniform mixing then the process is a mixingale, giving access to a law of large numbers. The above is applied to autoregessive processes of order 1 in Riesz spaces.
}
\parindent=0in
\parskip=.2in
\newsection{Introduction}
The 1962 paper of Ibragimov \cite[pages 370-371]{ibragimov} considers the variance and limiting cumulative distribution of functions of stationary strong mixing processes. Here the expectation of the deviation of the process from the time symmetric conditional expectations of the process with respect to the 
events in the time window is assumed to be summable over all window sizes.
Billingsley \cite[Section 21]{bill68} in his 1967 monograph developed on the ideas of Ibragimov. He worked in $L^2$, as opposed to $L^1$ used by Ibragimov, and was able to give a functional central limit theorem for such processes. Further relaxing the assumptions on these processes, McLeish 
\cite[page 837]{mcleish} was still able to obtain a strong law of large numbers for these processes as an application of his theory of mixingales.
In addition, it was McLeish who first used the term `epoque' in this context.
Gallant and White \cite{gallantwhite} generalized the theory to arrays of processes in $L^p$, they also introduced the term `near-epoch dependence' and presented a unified theory thereof.
We refer the reader to Davidson \cite[page 261]{davidson} for a further history of this development.

 For the reader's convenience we recall here the definition of near-epoch dependence given by Davidson in \cite{davidson}. 
A sequence of integrable random variables $(X_i)_{i\in \mathbb{Z}}$ in a probability space
$(\Omega,\mathcal{F}, \mathbb{P})$
is said to be near-epoch dependent
in $L^p$-norm
 on the sequence $(V_i)_{i\in \mathbb{Z}}$ in $(\Omega,\mathcal{F}, \mathbb{P})$
if, for $\mathcal{F}_{i-m}^{i+m}=\sigma(V_{i-m},\dots,V_{i+m})$, the $\sigma$-algebra generated by $V_{i-m}, \ldots, V_{i+m}$,
$$\|X_i-\mathbb{E}[X_i|\mathcal{F}_{i-m}^{i+m}]\|_p\le d_iv_m$$
where $v_m\to 0$ as $m\to\infty$ and $(d_i)_{i \in \mathbb{Z}}$ is a sequence of positive constants.
.

Schaefer \cite{schaefer}, Stoica \cite{stoica1, stoica2} and Troitsky \cite{troitsky} considered the extension of the concepts of conditional expectation and martingales to Banach lattices. See \cite{g-l, stoica-2015, troitskyx} for some recent developments in this area. 
In 2004, Kuo, Labuschagne and Watson \cite{discrete} generalized these concepts to vector lattices (Riesz spaces). de Pagter and Grobler in \cite{dPG} introduced the extension of a conditional expectation operator on a probability space to its natural domain. This extension was generalized to the Riesz space setting by Kuo, Labuschagne and Watson in \cite{condexp} and enabled the definition of the generalized $L^p$ spaces in \cite{trabelsi, mixing, stoc-int}. In particular,
the natural domain of a Riesz space conditional expectation operator, $T$, is
identified with the generalized $L^p$ space $\mathcal{L}^1(T)$ and $T|\cdot|$ defines
an $R(T)_+$ valued norm on  $\mathcal{L}^1(T)$, see \cite{mixing}, where 
$\mathcal{L}^\infty(T)$ and its $R(T)_+$ valued norm are also defined.
Building on this structure, using functional calculus, similar constructions can be made  for $p\in (1,\infty)$, see \cite{trabelsi}.
Critical in all of the above is that $R(T)$ is a universally complete $f$-algebra, see \cite{mixing}. The final piece needed in this context is Jensen's inequality in Riesz spaces, developed by Grobler in \cite{jensen's inequality}.

In \cite{ando-douglas} a Douglas-And\^o-Radon-Nikod\'ym theorem on Riesz spaces with a conditional expectation operator was proved. This gives the necessary tools for building and identifying conditional expectation operators on Riesz spaces, required in the study of Markov, \cite{vw,vw1} and mixing processes, \cite{mixing}. The study of mixing processes in Riesz spaces began with the the formulation of mixingales in Riesz spaces in \cite{mix}, where a law of large numbers was proved for Riesz space mixingales. In \cite{mixing} strong and uniform mixing processes were formulated in Riesz spaces and mixing inequalities proved for them.

In this work we introduce the concept of near-epoch dependence for Riesz space processes and show that if a process is near-epoch dependent and either strong or uniform mixing then the process is a mixingale, giving access again to a law of large numbers. Finally the above is applied to autoregessive processes of order 1 in Riesz spaces.

In Section 2 we provide the required background material on the generalized $L^p$ spaces. In Section 3 we recall for the readers' convenience the necessary results on mixingales and mixing processes in Riesz spaces.
In Section 4 we formulate near-epoch dependence for Riesz space processes and
prove that a near-epoch dependent process which is strong or uniform mixing is a mixingale, hence showing that such processes obey a law of large numbers.
In Section 5 we apply the results of Section 4 to autoregressive processes of order 1. These results lead to new results when applied to the classical probability setting, for this we refer the reader to the last section of \cite{mixing}.


\newsection{Preliminaries}

In this section we outline the prerequisite material relating to stochastic processes in Riesz spaces necessary for this paper. For a more thorough examination of the contents of this section, see, for example, \cite{discrete, condexp}. In addition, we will define the $\mathcal{L}^{2}(T)$ space and present the corresponding vector-valued norm, which will be defined with respect to the conditional expectation operator on the space.

\begin{defs}
\label{conditional expectation operator}
Let $E$ be a Dedekind complete Riesz space with weak order unit. A positive order continuous linear projection $T$ on $E$ with range $\mathcal{R}(T)$, a Dedekind complete Riesz subspace of $E$, is said to be a conditional expectation operator if $Te$ is a weak order unit of $E$ for each weak order unit $e$ of $E$.
\end{defs}

A Riesz space $E$ is said to be universally complete if $E$ is Dedekind complete and every subset of $E$ consisting of mutually disjoint elements has a supremum in $E$. If a Riesz space $E$ is not universally complete, we can define its universal completion $E^{u}$ as the universally complete Riesz space that contains $E$ as an order dense subspace. For a strictly positive conditional expectation operator $T$ on a Riesz space $E$, we denote by $\text{dom}(T)$ the maximal domain in $E^{u}$ to which $T$ can be extended as a conditional expectation operator, and we denote this extension again by $T$. A Dedekind complete Riesz space $E$ with conditional expectation operator $T$ and weak order unit $e = Te$ is said to be $T$-universally complete if $E = \text{dom}(T)$.

For a Dedekind complete Riesz space $E$ with conditional expectation operator $T$ and weak order unit $e = Te$, we denote the subspace of $e$-bounded elements by
\begin{align*}
E^{e} = \{f \in E : |f| \leq ke, \text{ for some } k \in \mathbb{R}_{+} \},
\end{align*}
which has a natural $f$-algebra structure. As noted in \cite{discrete}, we can extend the $f$-algebra structure in $E^{e}$ uniquely to the universal completion $E^{u}$.
This makes such an $f$-algebra structure accessible for
 $\text{dom}(T)$. 
 The multiplication so introduced is order continuous as the space
 is Archimedean, see \cite{zaanenII}. We recall from \cite{condexp} that a conditional expectation operator $T$ is an  averaging operator, i.e., if $f \in \mathcal{R}(T)$ and $g \in E$ such that $fg \in E$, then $T(fg) = fTg$. In addition, if $E$ is $T$-universally complete then $\mathcal{R}(T)$ is universally complete, as shown in \cite{mixing}.

\begin{defs}
\label{LpT spaces}
Let $E$ be a Dedekind complete Riesz space with conditional expectation operator $T$ and weak order unit $e = Te$. Define
\begin{enumerate}
\item[(i)] $\mathcal{L}^{1}(T) = \text{dom}(T)$,
\item[(ii)] $\mathcal{L}^{2}(T) = \{ f \in \mathcal{L}^{1}(T) : f^{2} \in \mathcal{L}^{1}(T) \}$,
\item[(iii)] $\mathcal{L}^{\infty}(T) = \{ f \in \mathcal{L}^{1}(T) : |f| \leq g, \text{ for some } g \in \mathcal{R}(T)_{+}\}$.
\end{enumerate}
\end{defs}

Note that if $f, g \in \mathcal{L}^{2}(T)$, then the product $fg \in \mathcal{L}^{1}(T)$. This is the case since $0 \leq (f \pm g)^{2} = f^{2} + g^{2} \pm 2fg$, which gives $2|fg| \leq f^{2} + g^{2} \in \mathcal{L}^{1}(T)$, and so the claim follows from the fact that $\mathcal{L}^{1}(T)$ is an order ideal in $E^{u}$. Recall from \cite{mixing} that the $\mathcal{L}^{1}(T)$ and $\mathcal{L}^{\infty}(T)$ spaces are $\mathcal{R}(T)$-modules. The analogous result for $\mathcal{L}^{2}(T)$ is presented as follows.

\begin{thm}
\label{Lp(T) are R(T)-modules}
$\mathcal{L}^{2}(T)$ is an $\mathcal{R}(T)$-module.
\end{thm}

\begin{proof}
Consider $f \in \mathcal{L}^{2}(T)_{+}$ and $g \in \mathcal{R}(T)$. Then $f^{2} \in \mathcal{L}^{1}(T)$ by definition, and $g^{2} \in \mathcal{R}(T)$ as $\mathcal{R}(T)$ is an $f$-algebra. Therefore, since $\mathcal{L}^{1}(T)$ is an $\mathcal{R}(T)$-module, $(fg)^{2} = f^{2} g^{2} \in \mathcal{L}^{1}(T)$, giving that $fg \in \mathcal{L}^{2}(T)$. \qed
\end{proof}

From \cite{mixing}, we have that the $T$-conditional norms $\Vert \cdot \Vert_{T,1} : f \mapsto T|f|$ and $\Vert \cdot \Vert_{T,\infty} : f \mapsto \inf\{g \in \mathcal{R}(T)_{+} : |f| \leq g\}$ define $\mathcal{R}(T)$-valued norms on $\mathcal{L}^{1}(T)$ and $\mathcal{L}^{\infty}(T)$, respectively. To define an analogous $T$-conditional norm on $\mathcal{L}^{2}(T)$, we require the following result.

\begin{lem}
\label{square root}
Let $E$ be a Dedekind complete Riesz space with weak order unit $e$. For all $f \in E_{+}$, there exists $\sqrt{f\,} \in E_{+}$ such that $\sqrt{f\,}^{\, 2} = f$.
\end{lem}

\begin{proof}
For $f \in E_{+}$, define, for each $n \in \mathbb{N}$,
\begin{align*}
f_{n} = \sum_{k=0}^{n2^{n} - 1} \frac{k}{2^{n}} P_{\left(\frac{k+1}{2^{n}}e - f\right)^{+}} \Big(I - P_{\left(\frac{k}{2^{n}}e - f\right)^{+}}\Big)e + n\big(I - P_{(ne - f)^{+}}\big)e,
\end{align*}
in which case $f_{n} \uparrow f$ in $E$ as $n \rightarrow \infty$. Then, for any band projection $P$ on $E$, recall that $Pe \cdot Pe = PPe = P^{2}e = Pe$, implying that $\sqrt{Pe} \, = Pe$. This, in combination with the fact that the summation above consists of mutually disjoint terms, gives that
\begin{align*}
\sqrt{f_{n}} \, = \sum_{k=0}^{n2^{n} - 1} \sqrt{\frac{k}{2^{n}}} \, P_{\left(\frac{k+1}{2^{n}}e - f\right)^{+}} \Big(I - P_{\left(\frac{k}{2^{n}}e - f\right)^{+}}\Big)e + \sqrt{n} \, \big(I - P_{(ne - f)^{+}}\big)e.
\end{align*}
It is easy to see that the sequence $(\sqrt{f_{n}} \, )_{n \in \mathbb{N}} \subset E_{+}$ is increasing. To obtain an upper bound, consider
\begin{align*}
P_{(e-f)^{+}} \sqrt{f_{n}} \, &= \sum_{k=0}^{2^{n} - 1} \sqrt{\frac{k}{2^{n}}} \, P_{\left(\frac{k+1}{2^{n}}e - f \right)^{+}} \Big(I - P_{\left(\frac{k}{2^{n}}e - f\right)^{+}}\Big)e \\
&\leq \sum_{k=0}^{2^{n} - 1} P_{\left(\frac{k+1}{2^{n}}e - f \right)^{+}} \Big(I - P_{\left(\frac{k}{2^{n}}e - f\right)^{+}}\Big)e \\
&= P_{(e-f)^{+}} e.
\end{align*}
On the other hand,
\begin{align*}
\big(I - P_{(e-f)^{+}}\big) \sqrt{f_{n}} \, = \, &\sum_{k=2^{n}}^{n2^{n}-1} \sqrt{\frac{k}{2^{n}}} \, P_{\left(\frac{k+1}{2^{n}}e - f \right)^{+}} \Big(I - P_{\left(\frac{k}{2^{n}}e - f\right)^{+}}\Big)e 
\, + \sqrt{n} \, \big(I - P_{(ne - f)^{+}}\big)e \\
\leq \, &\sum_{k=2^{n}}^{n2^{n}-1} \frac{k}{2^{n}} P_{\left(\frac{k+1}{2^{n}}e - f\right)^{+}} \Big(I - P_{\left(\frac{k}{2^{n}}e - f\right)^{+}}\Big)e + n \big(I - P_{(ne - f)^{+}}\big)e \\
\leq \, &\sum_{k=2^{n}}^{n2^{n}-1} P_{\left(\frac{k+1}{2^{n}}e - f\right)^{+}} \Big(I - P_{\left(\frac{k}{2^{n}}e - f\right)^{+}}\Big)f + \big(I - P_{(ne - f)^{+}}\big)f \\
= \, &\big(I - P_{(e - f)^{+}}\big)f.
\end{align*}
Therefore, for all $n \in \mathbb{N}$,
\begin{align*}
\sqrt{f_{n}} \, &= P_{(e-f)^{+}} \sqrt{f_{n}} \, + \big(I - P_{(e-f)^{+}}\big) \sqrt{f_{n}} 
\leq P_{(e-f)^{+}}e + \big(I - P_{(e-f)^{+}}\big)f \in E_{+}.
\end{align*}
Hence, we have that the increasing sequence $(\sqrt{f_{n}} \, )_{n \in \mathbb{N}} \subset E_{+}$ is bounded above in $E$, and so by the Dedekind completeness of $E$, there exists $g \in E_{+}$ such that $\sqrt{f_{n}} \, \uparrow g$ as $n \rightarrow \infty$. Finally, since $\sqrt{f_{n}}^{\, 2} = f_{n}$ for all $n \in \mathbb{N}$, $f_{n} \uparrow f$, and $\sqrt{f_{n}} \, \uparrow g$, we have that $g^{2} = f$, which follows from the order continuity of multiplication in $E^{u}$. As such, we can write $g = \sqrt{f\,} \,$. \qed
\end{proof}

From the preceding proof, it is apparent that if $0 \leq f \leq g$ in $E$, then $\sqrt{f } \, \leq \sqrt{g} \,$.

\begin{thm}
\label{riesz space norms}
Let $E$ be a $T$-universally complete Riesz space with weak order unit, where $T$ is a strictly positive conditional expectation operator on $E$. The map
\begin{align*}
f \mapsto \Vert f \Vert_{T,2} = \sqrt{T|f|^{2}}
\end{align*}
defines an $\mathcal{R}(T)_{+}$ valued norm on $\mathcal{L}^{2}(T)$.
\end{thm}

\begin{proof}
For $f \in \mathcal{L}^{2}(T)$, using the strict positivity of $T$,
\begin{align*}
f = 0 \Leftrightarrow \left|f\right|^{2} = 0 \Leftrightarrow T\left|f\right|^{2} = 0 \Leftrightarrow \sqrt{T|f|^{2}} \, = 0 \Leftrightarrow \left\Vert f \right\Vert_{T,2} = 0.
\end{align*}
For $f \in \mathcal{R}(T)$ and $g \in \mathcal{L}^{2}(T)$, we have $f^{2} \in \mathcal{R}(T)$ and $g^{2} \in \mathcal{L}^{1}(T)$, and as $T$ is an averaging operator,
\begin{align*}
\left\Vert fg \right\Vert_{T,2} = \sqrt{T|fg|^{2}} \, = \sqrt{T(|f|^{2} |g|^{2})} \, = \sqrt{|f|^{2} T|g|^{2}} \, = |f| \sqrt{T|g|^{2}} \, = |f| \left\Vert g \right\Vert_{T,2}.
\end{align*}
For the triangle inequality, we appeal to \cite{trabelsi}, in which functional calculus for convex mappings is used to prove the result. In particular, for $f, g \in \mathcal{L}^{2}(T)$,
\begin{align*}
\left\Vert f + g \right\Vert_{T,2} \leq \left\Vert f \right\Vert_{T,2} + \left\Vert g \right\Vert_{T,2}. \tag*{\qed}
\end{align*}
\end{proof}

The following is a special case of H\"{o}lder's inequality, which is proved in \cite{trabelsi, mixing}.

\begin{thm}
\label{Holder's inequality}
If $f \in \mathcal{L}^{p}(T)$ and $g \in \mathcal{L}^{q}(T)$, $(p,q) \in \{(1,\infty), (2,2)\}$, then
\begin{align*}
\Vert fg \Vert_{T,1} \leq \Vert f \Vert_{T,p} \Vert g \Vert_{T,q}.
\end{align*}
\end{thm}

Note that $\mathcal{L}^{\infty}(T) \subset \mathcal{L}^{2}(T) \subset \mathcal{L}^{1}(T)$. The following is a special case of Lyapunov's inequality.

\begin{thm}
\label{riesz space lyapunov's inequality}
If $f \in \mathcal{L}^{2}(T)$ and $g \in \mathcal{L}^{\infty}(T)$, then
\begin{align*}
\Vert f \Vert_{T,1} \leq \Vert f \Vert_{T,2} \text{ and } \Vert g \Vert_{T,2} \leq \Vert g \Vert_{T,\infty}.
\end{align*}
\end{thm}

\begin{proof}
The first inequality follows as a corollary to Theorem \ref{Holder's inequality} by setting $g=e$ therein. For the second inequality, let $g \in \mathcal{L}^{\infty}(T)$, then $|g| \leq h := \Vert g \Vert_{T,\infty} \in \mathcal{R}(T)_{+}$, so $|g|^{2} \leq h^{2} \in \mathcal{R}(T)$, since $\mathcal{R}(T)$ is universally complete. Then, as $T$ is positive, $T|g|^{2} \leq Th^{2} = h^{2}$. Therefore $\Vert g \Vert_{T,2} = \sqrt{T|g|^{2}} \, \leq \sqrt{h^{2}} = h = \Vert g \Vert_{T,\infty}$. \qed
\end{proof}

Next, we will prove a variant of the conditional Jensen's inequality. For additional details on conditional Jensen's inequalities in Riesz spaces, see \cite{jensen's inequality}.

\begin{thm}
\label{jensen's inequality}
Let $S$ be a conditional expectation operator on $\mathcal{L}^{1}(T)$ compatible with $T$ (in the sense that $ST = TS = T$), then for $f \in \mathcal{L}^{p}(T)$, $p \in \{1,2,\infty\}$,
\begin{align*}
\Vert Sf \Vert_{T,p} \leq \Vert f \Vert_{T,p}.
\end{align*}
\end{thm}

\begin{proof}
The result for $p \in \{1, \infty\}$ was proved in \cite{mixing}. Consider $f \in \mathcal{L}^{2}(T)$, then $Sf \in \mathcal{L}^{2}(T)$, as shown in \cite{bern}. By the positivity of $S$, we have $|Sf| \leq S|f|$, and by applying Theorem \ref{riesz space lyapunov's inequality} with $T$ replaced by $S$, we have $(S|f|)^{2} \leq S|f|^{2}$. Therefore
\begin{align*}
\Vert Sf \Vert_{T,2}^{2} = T|Sf|^{2} \leq T(S|f|)^{2} \leq TS|f|^{2} = T|f|^{2} = \Vert f \Vert_{T,2}^{2}. \tag*{\qed}
\end{align*}
\end{proof}

Similar results can be proved in the above manner for $p\in(1,\infty)$ and $q=1/p$.

\newsection{Mixing and mixingales}

In this section we list all pertinent results from \cite{mixing, mix} relating to the notions of mixing and mixingales. For the following definition, we denote by $\mathcal{B}(S)$, for a conditional expectation operator $S$ on a Riesz space $E$, the class of band projections on $E$ such that $P \in \mathcal{B}(S)$ if and only if $Pe \in \mathcal{R}(S)$.

\begin{defs}
\label{riesz space mixing coefficients}
Let $E$ be a Dedekind complete Riesz space with strictly positive conditional expectation operator $T$ and weak order unit $e = Te$, and let $U$ and $V$ be conditional expectation operators on $E$ compatible with $T$.
\begin{enumerate}
\item[(i)] The $T$-conditional strong mixing coefficient
\index{$T$-conditional mixing coefficient!strong}
between $U$ and $V$ is given by
\begin{align*}
\alpha_{T}(U,V) = \sup{\big\{ \left| TPQe - TPe \cdot TQe \right| : P \in \mathcal{B}(U), Q \in \mathcal{B}(V) \big\}}.
\end{align*}
\item[(ii)] The $T$-conditional uniform mixing coefficient
\index{$T$-conditional mixing coefficient!uniform}
between $U$ and $V$ is given by
\begin{align*}
\varphi_{T}(U,V) = \sup{\big\{ \left\Vert UQe - TQe \right\Vert_{T,\infty} : Q \in \mathcal{B}(V) \big\}}.
\end{align*}
\end{enumerate}
\end{defs}

The mixing coefficients outlined above can be interpreted as loosely measuring the dependence between the conditional expectation operators $U$ and $V$. In particular, if $U$ and $V$ are $T$-conditionally independent (in the sense that $UV = VU = T$), then it is easy to see that $\alpha_{T}(U,V) = \varphi_{T}(U,V) = 0$. The converse, however, holds only in the case of uniform mixing. This suggests that strong mixing provides a relatively weaker measure of independence, which is substantiated by the following result from \cite{mixing}.

\begin{lem}
\label{mixing coefficients inequality}
Let $E$ be a $T$-universally complete Riesz space with weak order unit $e = Te$, where $T$ is a strictly positive conditional expectation operator on $E$, and let $U$ and $V$ be conditional expectation operators on $E$ compatible with $T$, then
\begin{align*}
\alpha_{T}(U,V) \leq \varphi_{T}(U,V).
\end{align*}
\end{lem}

The following theorems relate to what are known as mixing inequalities, which were proved in \cite{mixing}, and are used to establish the relationship between mixingales and near-epoch dependence, see Theorem \ref{NED on mixing is mixingale}.

\begin{thm}
\label{alpha mixing inequality}
Let $E$ be a $T$-universally complete Riesz space with weak order unit $e=Te$, where $T$ is a strictly positive conditional expectation operator on $E$, and let $U$ and $V$ be conditional expectation operators on $E$ compatible with $T$. Then for $f \in \mathcal{R}(V) \cap \mathcal{L}^{\infty}(T)$,
\begin{align*}
\left\Vert Uf - Tf \right\Vert_{T,1} \leq 4 \alpha_{T}(U,V) \left\Vert f \right\Vert_{T,\infty}.
\end{align*}
\end{thm}

\begin{thm}
\label{phi mixing inequality}
Let $E$ be a $T$-universally complete Riesz space with weak order unit $e = Te$, where $T$ is a strictly positive conditional expectation operator on $E$, and let $U$ and $V$ be conditional expectation operators on $E$ compatible with $T$. Then for $f \in \mathcal{R}(V) \cap \mathcal{L}^{\infty}(T)$,
\begin{align*}
\left\Vert Uf - Tf \right\Vert_{T,1} \leq \left\Vert Uf - Tf \right\Vert_{T,\infty} \leq 2 \varphi_{T}(U,V) \left\Vert f \right\Vert_{T,\infty}.
\end{align*}
\end{thm}

Next, we provide a Riesz space analogue of mixing processes.
This requires an extension of Definition \ref{riesz space mixing coefficients} to sequences.

\begin{defs}
\label{riesz space mixing sequences}
Let $E$ be a $T$-universally complete Riesz space with weak order unit $e = Te$, where $T$ is a strictly positive conditional expectation operator on $E$, and let $(T_{n})_{n \in \mathbb{Z}}$ be a sequence of conditional expectation operators on $E$ compatible with $T$. Define, for $m \in \mathbb{N}$,
\begin{enumerate}
\item[(i)] $\alpha_{T,m} = \sup{\{\alpha_{T}(T_{-\infty}^{n},T_{\, n+m}^{\infty}) : n \in \mathbb{Z} \}}$,
\item[(ii)] $\varphi_{T,m} = \sup{\{\varphi_{T}(T_{-\infty}^{n},T_{\, n+m}^{\infty}) : n \in \mathbb{Z} \}}$,
\end{enumerate}
where $T_{-\infty}^{n}$ and $T_{\, n+m}^{\infty}$ are conditional expectation operators on $E$ compatible with $T$ having $\mathcal{R}(T_{-\infty}^{n}) = \left\langle \cup_{i \leq n} \mathcal{R}(T_{i}) \right\rangle$ and $\mathcal{R}(T_{\, n+m}^{\infty}) = \left\langle \cup_{i \geq n+m} \mathcal{R}(T_{i}) \right\rangle$. Here $\left\langle \cup_{i \leq n} \mathcal{R}(T_{i}) \right\rangle$ and $\left\langle \cup_{i \geq n+m} \mathcal{R}(T_{i}) \right\rangle$ denote the smallest Riesz subspaces of $E$ containing $\cup_{i \leq n} \mathcal{R}(T_{i})$ and $\cup_{i \geq n+m} \mathcal{R}(T_{i})$, respectively, which are closed with respect to order limits in $E$.
\end{defs}

\begin{defs}
\label{riesz space mixing operators}
Let $E$ be a $T$-universally complete Riesz space with weak order unit $e = Te$, where $T$ is a strictly positive conditional expectation operator on $E$. The sequence $(T_{n})_{n \in \mathbb{Z}}$ of conditional expectation operators on $E$ compatible with $T$ is said to be $\alpha_{T}$-mixing ($\varphi_{T}$-mixing) if $\alpha_{T,m} \rightarrow 0$ ($\varphi_{T,m} \rightarrow 0$) in order as $m \rightarrow \infty$.
\end{defs}

\begin{defs}
\label{riesz space mixing processes}
Let $E$ be a $T$-universally complete Riesz space with weak order unit $e = Te$, where $T$ is a strictly positive conditional expectation operator on $E$. The sequence $(f_{n})_{n \in \mathbb{Z}} \subset E$ is said to be $\alpha_{T}$-mixing ($\varphi_{T}$-mixing) if the sequence of conditional expectation operators $(T_{n})_{n \in \mathbb{Z}}$ with $\mathcal{R}(T_{n}) = \langle \{f_{n}\} \cup \mathcal{R}(T) \rangle$ is $\alpha_{T}$-mixing ($\varphi_{T}$-mixing).
\end{defs}

The existence of the conditional expectation operators $T_{-\infty}^{n}$, $T_{n+m}^{\infty}$ and $T_{n}$ in the preceding definitions follows directly from \cref{ando-douglas}. 
Also, in view of Lemma \ref{mixing coefficients inequality}, it is the case that a $\varphi_{T}$-mixing sequence is necessarily $\alpha_{T}$-mixing. 
As discussed in \cite{gallantwhite}, mixing sequences can be used to describe the time dependent location of a particle suspended in a mixture. In particular, how the position becomes progressively less dependent on its initial position as time progresses.
Definition \ref{riesz space mixing processes} states that the future realizations of a mixing sequence $(f_{n})_{n \in \mathbb{Z}}$ are loosely $T$-conditionally independent of the past realizations, as measured by the relevant mixing coefficient, and that this independence strengthens as the gap between the past, present and future widens.

We now review the theory of mixingales in Riesz spaces. Recall that a sequence $(T_{n})_{n \in \mathbb{Z}}$ of conditional expectation operators on a Riesz space $E$ is said to be a filtration on $E$ if $T_{m}T_{n} = T_{n}T_{m} = T_{m}$ for all $m \leq n$. The following definition, which is extended from \cite{mix}, provides Riesz space analogues of the classical $\mathcal{L}^{p}$-mixingales for $p \in \{1,2, \infty\}$.

\begin{defs}
\label{mixingale}
Let $E$ be a $T$-universally complete Riesz space with weak order unit $e=Te$, where $T$ is a strictly positive conditional expectation operator on $E$. Let $(T_{n})_{n \in \mathbb{Z}}$ be a filtration on $E$ compatible with $T$ and $(f_{n})_{n \in \mathbb{Z}} \subset \mathcal{L}^{p}(T)$, $p \in \{1, 2, \infty\}$. The double sequence $(f_{n},T_{n})_{n \in \mathbb{Z}}$ is said to be a mixingale
in $\mathcal{L}^{p}(T)$ if there exist sequences $(c_{n})_{n \in \mathbb{Z}} \subset \mathcal{L}^{p}(T)_{+}$ and $(\phi_{m})_{m \in \mathbb{N}} \subset \mathcal{R}(T)_{+}$ with $\phi_{m} \rightarrow 0$ in order as $m \rightarrow \infty$, and for all $n \in \mathbb{Z}$ and $m \in \mathbb{N}$, we have
\begin{enumerate}
\item[(i)] $\left\Vert T_{n-m}f_{n} \right\Vert_{T,p} \leq c_{n}\phi_{m}$,
\item[(ii)] $\left\Vert f_{n} - T_{n+m}f_{n} \right\Vert_{T,p} \leq c_{n}\phi_{m+1}$.
\end{enumerate}
\end{defs}

The extension from \cite{mix} represented in the preceding definition relates to the mixingale numbers $(\phi_{m})_{m \in \mathbb{N}}$ being elements of $\mathcal{R}(T)_{+}$ as opposed to $\mathbb{R}_{+}$, as well as to the consideration of the cases $p \in \{2, \infty\}$. Note that, since the $\mathcal{L}^{p}(T)$ spaces are $\mathcal{R}(T)$-modules, the bounds $c_{n}\phi_{m}$ and $c_{n}\phi_{m+1}$ in the preceding definition are in $\mathcal{L}^{p}(T)$, for $p \in \{1, 2, \infty\}$. Also, for $p = \infty$, we have
\begin{enumerate}
\item[(i)] $\left\Vert T_{n-m}f_{n}\right\Vert_{T,\infty} \leq c_{n} \phi_{m} \Rightarrow \left|T_{n-m}f_{n}\right| \leq c_{n} \phi_{m}$,
\item[(ii)] $\left\Vert f_{n} - T_{n-m}f_{n}\right\Vert_{T,\infty} \leq c_{n} \phi_{m+1} \Rightarrow \left|f_{n} - T_{n+m}f_{n}\right| \leq c_{n} \phi_{m+1}$.
\end{enumerate}

To complete this section, we provide an extension of the main result proved in \cite{mix}.

\begin{defs}
\label{T-uniform}
Let $E$ be a Dedekind complete Riesz space with conditional expectation operator $T$ and weak order unit $e = Te$. The net $(f_{\alpha})_{\alpha \in \Lambda}$ in $E$, where $\Lambda$ is some index set, is said to be $T$-uniform if
\begin{align*}
\sup\{TP_{(|f_{\alpha}| - ce)^{+}}|f_{\alpha}| : \alpha \in \Lambda\} \rightarrow 0 \text{ in order as } c \rightarrow \infty.
\end{align*}
\end{defs}

\begin{thm}
\label{mixingale weak law of large numbers}
Let $E$ be a $T$-universally complete Riesz space with weak order unit $e=Te$, where $T$ is a strictly positive conditional expectation operator on $E$, and let $(f_{n}, T_{n})_{n \in \mathbb{Z}}$ be a $T$-uniform mixingale in $\mathcal{L}^{p}(T)$, $p \in \{1,2,\infty\}$, with $(c_{n})_{n \in \mathbb{Z}}$ and $(\phi_{m})_{m \in \mathbb{N}}$ as defined in Definition \ref{mixingale}.
\begin{enumerate}
\item[(i)] If $\displaystyle{\left(\frac{1}{m} \sum_{i=n+1}^{n+m} c_{i} \right)_{m \in \mathbb{N}}}$ is bounded in $E$ uniformly in $n \in \mathbb{Z}$, then
\begin{align*}
T|\overline{f}_{n,m}| = T\left|\frac{1}{m} \sum_{i=n+1}^{n+m} f_{i} \right| \rightarrow 0 \text{ in order as } m \rightarrow \infty, \text{ uniformly in } n \in \mathbb{Z}.
\end{align*}
\item[(ii)] If $c_{n} = T|f_{n}|$ for all $n \in \mathbb{Z}$, then
\begin{align*}
T|\overline{f}_{n,m}| = T\left|\frac{1}{m} \sum_{i=n+1}^{n+m} f_{i} \right| \rightarrow 0 \text{ in order as } m \rightarrow \infty, \text{ uniformly in } n \in \mathbb{Z}.
\end{align*}
\end{enumerate}
\end{thm}

Proof of the preceding theorem is largely unchanged from that provided in \cite{mix}. Firstly, the case of $p \in \{2, \infty\}$ follows immediately from that of $p=1$, since mixingales in $\mathcal{L}^{2}(T)$ and $\mathcal{L}^{\infty}(T)$ are mixingales in $\mathcal{L}^{1}(T)$, by Lyapunov's inequality. 
Secondly, the condition that convergence occurs uniformly in $n \in \mathbb{Z}$ is a result of the boundedness being uniform in $n \in \mathbb{Z}$. 
Thirdly, the generalization of the proof from the mixingale numbers being in $\mathbb{R}_{+}$ to them being in $\mathcal{R}(T)_{+}$ follows from $\mathcal{L}^{1}(T)$ being an $\mathcal{R}(T)$-module and $\mathcal{R}(T)$ being a universally complete $f$-algebra.


\newsection{Near-epoch dependence}

We recall the classical definition of near-epoch dependence, which follows primarily from \cite{davidson}.

\begin{defs}
\label{measure theory NED}
Let $(Y_{n})_{n \in \mathbb{Z}}$ be a stochastic process defined on a probability space $(\Omega, \mathcal{F}, \mathbb{P})$. The sequence of integrable random variables $(X_{n})_{n \in \mathbb{Z}}$ is said to be near-epoch dependent in $L^{p}$-norm, for $1 \leq p \leq \infty$, on $(Y_{n})_{n \in \mathbb{Z}}$ if there exist sequences $(d_{n})_{n \in \mathbb{Z}}, (\xi_{m})_{m \in \mathbb{N}} \subset \mathbb{R}_{+}$ with $\xi_{m} \rightarrow 0$ as $m \rightarrow \infty$, and for all $n \in \mathbb{Z}$ and $m \in \mathbb{N}$, we have
\begin{align*}
\Vert X_{n} - \mathbb{E}(X_{n} \, \vert \, \mathcal{F}_{n-m}^{n+m})\Vert_{p} \leq d_{n} \xi_{m},
\end{align*}
where $\mathcal{F}_{n-m}^{n+m} = \sigma(Y_{i} : i = n-m, \ldots, n+m)$.
\end{defs}

From a conceptual viewpoint, near-epoch dependence should not be interpreted as a property of the random variables $(Y_{n})_{n \in \mathbb{Z}}$ and $(X_{n})_{n \in \mathbb{Z}}$, but rather as a property of the mapping between them. In particular, near-epoch dependence provides a suitably generalized means of studying the characteristics of stochastic systems in which a sequence of dependent random variables $(X_{n})_{n \in \mathbb{Z}}$ depends primarily on the near incidents of a sequence of explanatory random variables $(Y_{n})_{n \in \mathbb{Z}}$. As such, a natural line of study relates to the determination of the properties induced on the dependent process by the explanatory process, an example of which is presented in Theorem \ref{NED on mixing is mixingale}.

In preparation for the Riesz space characterization of near-epoch dependence, note that the family of sub-$\sigma$-algebras $(\mathcal{F}_{m}^{n})$ generated by $(Y_{m}, \ldots, Y_{n})$ in the preceding definition can be characterized as satisfying
\begin{align*}
\mathcal{F}_{m+1}^{n} \subset \mathcal{F}_{m}^{n} \subset \mathcal{F}_{m}^{n+1},
\end{align*}
for all $m, n \in \mathbb{Z}$ such that $m < n$. Using this characterization directly, it is possible to formulate an abstract definition of near-epoch dependence without any reference to the underlying stochastic process $(Y_{n})_{n \in \mathbb{Z}}$.

For brevity, it is assumed throughout the remainder of this section, unless otherwise stated, that $E$ is a $T$-universally complete Riesz space with weak order unit $e = Te$, where $T$ is a strictly positive conditional expectation operator on $E$. The following definition provides the Riesz space analogue of near-epoch dependence in $L^{p}$-norm for $p \in \{1, 2, \infty\}$. 

\begin{defs}
\label{NED}
Let $(T_{i}^{j})$ be a family of conditional expectation operators on $E$ compatible with $T$ such that $\mathcal{R}(T_{i+1}^{j}) \subset \mathcal{R}(T_{i}^{j}) \subset \mathcal{R}(T_{i}^{j+1})$ for all $-\infty < i < j < \infty$. The sequence $(f_{n})_{n \in \mathbb{Z}} \subset \mathcal{L}^{p}(T)$, $p \in \{1, 2, \infty\}$, is said to be near-epoch dependent
in $\mathcal{L}^{p}(T)$ on $(T_{i}^{j})$ if there exist sequences $(d_{n})_{n \in \mathbb{Z}} \subset \mathcal{L}^{p}(T)_{+}$ and $(\xi_{m})_{m \in \mathbb{N}} \subset \mathcal{R}(T)_{+}$ with $\xi_{m} \rightarrow 0$ in order as $m \rightarrow \infty$, and for all $n \in \mathbb{Z}$ and $m \in \mathbb{N}$, we have
\begin{align*}
\Vert f_{n} - T_{n-m}^{n+m}f_{n} \Vert_{T,p} \leq d_{n} \xi_{m}.
\end{align*}
\end{defs}

As in the case of mixingales in $\mathcal{L}^{\infty}(T)$, we have, for the near-epoch dependence property
\begin{align*}
\Vert f_{n} - T_{n-m}^{n+m}f_{n} \Vert_{T, \infty} \leq d_{n} \xi_{m} \Rightarrow |f_{n} - T_{n-m}^{n+m}f_{n}| \leq d_{n} \xi_{m}.
\end{align*}

Again as in the case of mixingales, near-epoch dependence in $\mathcal{L}^{\infty}(T)$ implies near-epoch dependence in $\mathcal{L}^{2}(T)$, and near-epoch dependence in $\mathcal{L}^{2}(T)$ implies near-epoch dependence in $\mathcal{L}^{1}(T)$.

We will now prove several elementary theorems related to sums, products and shifts of near-epoch dependent sequences. It will be assumed throughout the following that $(T_{i}^{j})$ is a family of conditional expectation operators on $E$ compatible with $T$ such that $\mathcal{R}(T_{i+1}^{j}) \subset \mathcal{R}(T_{i}^{j}) \subset \mathcal{R}(T_{i}^{j+1})$ for all $-\infty < i < j < \infty$.

\begin{thm}
\label{NED addition}
Let $(f_{n})_{n \in \mathbb{Z}}$ and $(g_{n})_{n \in \mathbb{Z}}$ be near-epoch dependent in $\mathcal{L}^{p}(T)$, $p \in \{1,2,\infty\}$, on $(T_{i}^{j})$. Then $(f_{n}+g_{n})_{n \in \mathbb{Z}}$ is near-epoch dependent in $\mathcal{L}^{p}(T)$ on $(T_{i}^{j})$.
\end{thm}

\begin{proof}
Since $(f_{n})_{n \in \mathbb{Z}}$ and $(g_{n})_{n \in \mathbb{Z}}$ are near-epoch dependent in $\mathcal{L}^{p}(T)$ on $(T_{i}^{j})$, there are sequences $(d_{n}^{f})_{n \in \mathbb{Z}}, (d_{n}^{g})_{n \in \mathbb{Z}} \subset \mathcal{L}^{p}(T)_{+}$ and $(\xi_{m}^{f})_{m \in \mathbb{N}}, (\xi_{m}^{g})_{m \in \mathbb{N}} \subset \mathcal{R}(T)_{+}$ with $\xi_{m}^{f} \rightarrow 0$ and $\xi_{m}^{g} \rightarrow 0$ in order as $m \rightarrow \infty$, and for all $n \in \mathbb{Z}$ and $m \in \mathbb{N}$,
\begin{gather*}
\Vert f_{n} - T_{n-m}^{n+m}f_{n} \Vert_{T,p} \leq d_{n}^{f} \xi_{m}^{f}, \\
\Vert g_{n} - T_{n-m}^{n+m}g_{n} \Vert_{T,p} \leq d_{n}^{g} \xi_{m}^{g}.
\end{gather*}
By the triangle inequality for $T$-conditional norms,
\begin{align*}
\Vert (f_{n}+g_{n}) - T_{n-m}^{n+m}(f_{n}+g_{n}) \Vert_{T,p} &= \Vert f_{n} - T_{n-m}^{n+m}f_{n} + g_{n} - T_{n-m}^{n+m}g_{n}\Vert_{T,p} \\
&\leq \Vert f_{n} - T_{n-m}^{n+m}f_{n} \Vert_{T,p} + \Vert g_{n} - T_{n-m}^{n+m}g_{n} \Vert_{T,p} \\
&\leq d_{n}^{f} \xi_{m}^{f} + d_{n}^{g} \xi_{m}^{g} \\
&\leq (d_{n}^{f} \vee d_{n}^{g})(\xi_{m}^{f} + \xi_{m}^{g}) \\
&= d_{n} \xi_{m},
\end{align*}
where $d_{n} = d_{n}^{f} \vee d_{n}^{g} \in \mathcal{L}^{p}(T)_{+}$ for all $n \in \mathbb{Z}$, and $\xi_{m} = \xi_{m}^{f} + \xi_{m}^{g} \rightarrow 0$ in order as $m \rightarrow \infty$, giving that $(f_{n}+g_{n})_{n \in \mathbb{Z}}$ is near-epoch dependent in $\mathcal{L}^{p}(T)$ on $(T_{i}^{j})$. \qed
\end{proof}

\begin{cor}
\label{NED addition corollary}
Let $(f_{n})_{n \in \mathbb{Z}}$ and $(g_{n})_{n \in \mathbb{Z}}$ be near-epoch dependent in $\mathcal{L}^{p}(T)$ and $\mathcal{L}^{q}(T)$, $p,q \in \{1,2,\infty\}$, respectively, on $(T_{i}^{j})$. Then $(f_{n}+g_{n})_{n \in \mathbb{Z}}$ is near-epoch dependent in $\mathcal{L}^{r}(T)$ on $(T_{\, i}^{j})$, for $r = \min\left\{p,q\right\}$.
\end{cor}

\begin{thm}
\label{NED product 1 infinity and infinity}
Let $(f_{n})_{n \in \mathbb{Z}}$ and $(g_{n})_{n \in \mathbb{Z}}$ be near-epoch dependent in $\mathcal{L}^{p}(T)$, $p \in \{1,\infty\}$, and $\mathcal{L}^{\infty}(T)$, respectively, on $(T_{i}^{j})$. Then $(f_{n}g_{n})_{n \in \mathbb{Z}}$ is near-epoch dependent in $\mathcal{L}^{p}(T)$ on $(T_{i}^{j})$.
\end{thm}

\begin{proof}
Before proceeding, note that $(f_{n}g_{n})_{n \in \mathbb{Z}} \subset \mathcal{L}^{p}(T)$. Since $(f_{n})_{n \in \mathbb{Z}}$ and $(g_{n})_{n \in \mathbb{Z}}$ are near-epoch dependent in $\mathcal{L}^{p}(T)$ and $\mathcal{L}^{\infty}(T)$, respectively, on $(T_{i}^{j})$, there are sequences $(d_{n}^{f})_{n \in \mathbb{Z}} \subset \mathcal{L}^{p}(T)_{+}$, $(d_{n}^{g})_{n \in \mathbb{Z}} \subset \mathcal{L}^{\infty}(T)_{+}$ and $(\xi_{m}^{f})_{m \in \mathbb{N}}, (\xi_{m}^{g})_{m \in \mathbb{N}} \subset \mathcal{R}(T)_{+}$ with $\xi_{m}^{f} \rightarrow 0$ and $\xi_{m}^{g} \rightarrow 0$ in order as $m \rightarrow \infty$, and for all $n \in \mathbb{Z}$ and $m \in \mathbb{N}$,
\begin{gather*}
\Vert f_{n} - T_{n-m}^{n+m}f_{n} \Vert_{T,p} \leq d_{n}^{f} \xi_{m}^{f}, \\
|g_{n} - T_{n-m}^{n+m}g_{n}| \leq d_{n}^{g} \xi_{m}^{g}.
\end{gather*}
By the triangle inequality for $T$-conditional norms,
\begin{align*}
\Vert f_{n}g_{n} - T_{n-m}^{n+m}f_{n}g_{n} \Vert_{T,p} = &\Vert (f_{n}g_{n} - f_{n}T_{n-m}^{n+m}g_{n}) + (f_{n}T_{n-m}^{n+m}g_{n} - (T_{n-m}^{n+m}f_{n})(T_{n-m}^{n+m}g_{n})) \\
&- T_{n-m}^{n+m}[(f_{n} - T_{n-m}^{n+m}f_{n})(g_{n} - T_{n-m}^{n+m}g_{n})] \Vert_{T,p} \\
\leq \, &\Vert f_{n}(g_{n} - T_{n-m}^{n+m}g_{n}) \Vert_{T,p} \\
&+ \Vert (f_{n} - T_{n-m}^{n+m}f_{n})T_{n-m}^{n+m}g_{n} \Vert_{T,p} \\
&+ \Vert T_{n-m}^{n+m}[(f_{n} - T_{n-m}^{n+m}f_{n})(g_{n} - T_{n-m}^{n+m}g_{n})] \Vert_{T,p}.
\end{align*}
For the first term, since $d_{n}^{g} \in \mathcal{L}^{\infty}(T)_{+}$, there exists $r_{n}^{g} \in \mathcal{R}(T)_{+}$ such that $d_{n}^{g} \leq r_{n}^{g}$. By the monotonicity and homogeneity of $T$-conditional norms,
\begin{align*}
\Vert f_{n}(g_{n} - T_{n-m}^{n+m}g_{n}) \Vert_{T,p} &\leq \Vert |f_{n}| \cdot d_{n}^{g} \xi_{m}^{g} \Vert_{T,p} \\
&\leq \xi_{m}^{g} \Vert f_{n} r_{n}^{g} \Vert_{T,p} \\
&= \xi_{m}^{g} r_{n}^{g} \Vert f_{n} \Vert_{T,p}.
\end{align*}
For the second term, since $g_{n} \in \mathcal{L}^{\infty}(T)$, there exists $h_{n} \in \mathcal{R}(T)_{+}$ such that $|g_{n}| \leq h_{n}$. By the positivity of $T_{n-m}^{n+m}$ and its compatibility with $T$, and as $h_{n} \in \mathcal{R}(T)$,
\begin{align*}
|T_{n-m}^{n+m} g_{n}| \leq T_{n-m}^{n+m} |g_{n}| \leq T_{n-m}^{n+m} h_{n} = T_{n-m}^{n+m} T h_{n} = Th_{n} = h_{n}.
\end{align*}
Again by the monotonicity and homogeneity of $T$-conditional norms,
\begin{align*}
\Vert (f_{n} - T_{n-m}^{n+m}f_{n})T_{n-m}^{n+m}g_{n} \Vert_{T,p} &\leq \Vert |f_{n} - T_{n-m}^{n+m}f_{n}| h_{n} \Vert_{T,p} \\
&= h_{n} \Vert f_{n} - T_{n-m}^{n+m}f_{n} \Vert_{T,p} \\
&\leq h_{n} d_{n}^{f} \xi_{m}^{f}.
\end{align*}
For the third term, it follows from Jensen's inequality that
\begin{align*}
\Vert T_{n-m}^{n+m}[(f_{n} - T_{n-m}^{n+m}f_{n})(g_{n} - T_{n-m}^{n+m}g_{n})] \Vert_{T,p} &\leq \Vert (f_{n} - T_{n-m}^{n+m}f_{n})(g_{n} - T_{n-m}^{n+m}g_{n}) \Vert_{T,p} \\
&\leq \Vert |f_{n} - T_{n-m}^{n+m}f_{n}| d_{n}^{g} \xi_{m}^{g} \Vert_{T,p} \\
&\leq \xi_{m}^{g} \Vert |f_{n} - T_{n-m}^{n+m}f_{n}| r_{n}^{g} \Vert_{T,p} \\
&= \xi_{m}^{g} r_{n}^{g} \Vert f_{n} - T_{n-m}^{n+m} f_{n} \Vert_{T,p} \\
&\leq \xi_{m}^{g} r_{n}^{g} d_{n}^{f} \xi_{m}^{f}.
\end{align*}
Putting all of the above inequalities together, we have, for all $n \in \mathbb{Z}$ and $m \in \mathbb{N}$,
\begin{align*}
\Vert f_{n}g_{n} - T_{n-m}^{n+m}f_{n}g_{n} \Vert_{T,p} &\leq r_{n}^{g} \Vert f_{n} \Vert_{T,p} \xi_{m}^{g} + d_{n}^{f} h_{n} \xi_{m}^{f} + d_{n}^{f} r_{n}^{g} \xi_{m}^{f} \xi_{m}^{g} \\
&\leq (r_{n}^{g} \Vert f_{n} \Vert_{T,p} \vee d_{n}^{f} h_{n} \vee d_{n}^{f} r_{n}^{g})(\xi_{m}^{f} + \xi_{m}^{g} + \xi_{m}^{f} \xi_{m}^{g}) \\
&= d_{n} \xi_{m},
\end{align*}
where $d_{n} = r_{n}^{g} \left\Vert f_{n} \right\Vert_{T,p} \vee d_{n}^{f} h_{n} \vee d_{n}^{f} r_{n}^{g} \in \mathcal{L}^{p}(T)_{+}$ and $\xi_{m} = \xi_{m}^{f} + \xi_{m}^{g} + \xi_{m}^{f} \xi_{m}^{g} \rightarrow 0$ in order as $m \rightarrow \infty$, giving that $(f_{n}g_{n})_{n \in \mathbb{Z}}$ is near-epoch dependent in $\mathcal{L}^{p}(T)$ on $(T_{i}^{j})$. \qed
\end{proof}

\begin{thm}
\label{NED product 2 and 2}
Let $(f_{n})_{n \in \mathbb{Z}}$ and $(g_{n})_{n \in \mathbb{Z}}$ be near-epoch dependent in $\mathcal{L}^{2}(T)$ on $(T_{i}^{j})$. Then $(f_{n}g_{n})_{n \in \mathbb{Z}}$ is near-epoch dependent in $\mathcal{L}^{1}(T)$ on $(T_{i}^{j})$.
\end{thm}

\begin{proof}
Since $(f_{n})_{n \in \mathbb{Z}}$ and $(g_{n})_{n \in \mathbb{Z}}$ are near-epoch dependent in $\mathcal{L}^{2}(T)$ on $(T_{i}^{j})$, there are sequences $(d_{n}^{f})_{n \in \mathbb{Z}}, (d_{n}^{g})_{n \in \mathbb{Z}} \subset \mathcal{L}^{2}(T)_{+}$ and $(\xi_{m}^{f})_{m \in \mathbb{Z}}, (\xi_{m}^{g})_{m \in \mathbb{Z}} \subset \mathcal{R}(T)_{+}$ with $\xi_{m}^{f} \rightarrow 0$ and $\xi_{m}^{g} \rightarrow 0$ in order as $m \rightarrow \infty$, and for all $n \in \mathbb{Z}$ and $m \in \mathbb{N}$,
\begin{gather*}
\Vert f_{n} - T_{n-m}^{n+m}f_{n} \Vert_{T,2} \leq d_{n}^{f} \xi_{m}^{f}, \\
\Vert g_{n} - T_{n-m}^{n+m}g_{n} \Vert_{T,2} \leq d_{n}^{g} \xi_{m}^{g},
\end{gather*}
Then, carrying out the same manipulation from the preceding proof and using the triangle inequality for $T$-conditional norms, we have
\begin{align*}
\Vert f_{n}g_{n} - T_{n-m}^{n+m}f_{n}g_{n} \Vert_{T,1} \leq &\Vert f_{n}(g_{n} - T_{n-m}^{n+m}g_{n}) \Vert_{T,1} \\
&+ \Vert (f_{n} - T_{n-m}^{n+m}f_{n})T_{n-m}^{n+m}g_{n} \Vert_{T,1} \\
&+ \Vert T_{n-m}^{n+m}[(f_{n} - T_{n-m}^{n+m}f_{n})(g_{n} - T_{n-m}^{n+m}g_{n})] \Vert_{T,1}.
\end{align*}
For the first term, since $g_{n} - T_{\, n-m}^{n+m}g_{n} \in \mathcal{L}^{2}(T)$, we can apply H\"{o}lder's inequality,
\begin{align*}
\Vert f_{n}(g_{n} - T_{n-m}^{n+m}g_{n}) \Vert_{T,1} &\leq \Vert f_{n} \Vert_{T,2} \Vert g_{n} - T_{n-m}^{n+m}g_{n} \Vert_{T,2} \\
&\leq \Vert f_{n} \Vert_{T,2} d_{n}^{g} \xi_{m}^{g}.
\end{align*}
For the second term, since $f_{n} - T_{n-m}^{n+m}f_{n} \in \mathcal{L}^{2}(T)$, we can again apply H\"{o}lder's inequality and then Jensen's inequality,
\begin{align*}
\Vert (f_{n} - T_{n-m}^{n+m} f_{n}) T_{n-m}^{n+m}g_{n} \Vert_{T,1} &\leq \Vert f_{n} - T_{n-m}^{n+m} f_{n} \Vert_{T,2} \Vert T_{n-m}^{n+m} g_{n} \Vert_{T,2} \\
&\leq d_{n}^{f} \xi_{m}^{f} \Vert g_{n} \Vert_{T,2}.
\end{align*}
For the third term, it follows from H\"{o}lder's inequality and Jensen's inequality that
\begin{align*}
\Vert T_{n-m}^{n+m}[(f_{n} - T_{n-m}^{n+m}f_{n})(g_{n} - T_{n-m}^{n+m}g_{n})] \Vert_{T,1} &\leq \Vert (f_{n} - T_{n-m}^{n+m}f_{n})(g_{n} - T_{n-m}^{n+m}g_{n}) \Vert_{T,1} \\
&\leq \Vert f_{n} - T_{n-m}^{n+m}f_{n} \Vert_{T,2} \Vert g_{n} - T_{n-m}^{n+m}g_{n} \Vert_{T,2} \\
&\leq d_{n}^{f} \xi_{m}^{f} d_{n}^{g} \xi_{m}^{g}.
\end{align*}
Putting all of the above inequalities together, we have, for all $n \in \mathbb{Z}$ and $m \in \mathbb{N}$,
\begin{align*}
\Vert f_{n}g_{n} - T_{n-m}^{n+m}f_{n}g_{n} \Vert_{T,1} &\leq \Vert f_{n} \Vert_{T,2} d_{n}^{g} \xi_{m}^{g} + d_{n}^{f} \Vert g_{n} \Vert_{T,2} \xi_{m}^{f} + d_{n}^{f} d_{n}^{g} \xi_{m}^{f} \xi_{m}^{g} \\
&\leq (d_{n}^{g} \Vert f_{n} \Vert_{T,2} \vee d_{n}^{f} \Vert g_{n} \Vert_{T,2} \vee d_{n}^{f} d_{n}^{g}) (\xi_{m}^{f} + \xi_{m}^{g} + \xi_{m}^{f} \xi_{m}^{g}) \\
&= d_{n} \xi_{m},
\end{align*}
where $d_{n} = d_{n}^{g} \Vert f_{n} \Vert_{T,2} \vee d_{n}^{f} \Vert g_{n} \Vert_{T,2} \vee d_{n}^{f} d_{n}^{g} \in \mathcal{L}^{1}(T)_{+}$ and $\xi_{m} = \xi_{m}^{f} + \xi_{m}^{g} + \xi_{m}^{f} \xi_{m}^{g} \rightarrow 0$ in order as $m \rightarrow \infty$, giving that $(f_{n}g_{n})_{n \in \mathbb{Z}}$ is near-epoch dependent in $\mathcal{L}^{1}(T)$ on $(T_{\, i}^{j})$. \qed
\end{proof}

For proof of the final elementary theorem, which relates to shifts in near-epoch dependent sequences, we require the following lemma.

\begin{lem}
\label{NED inequality lemma}
Let $f \in \mathcal{L}^{p}(T)$, $p \in \{1,2, \infty\}$, and $U$ and $V$ be conditional expectation operators on $E$ compatible with $T$. If $\mathcal{R}(U) \subset \mathcal{R}(V)$, then
\begin{align*}
\Vert f - Vf \Vert_{T,p} \leq 2\Vert f - Uf \Vert_{T,p}.
\end{align*}
\end{lem}

\begin{proof}
Let $g = f - Uf$, then since $Uf \in \mathcal{R}(U) \subset \mathcal{R}(V)$,
\begin{align*}
g - Vg = f - Uf - Vf + VUf = f - Uf - Vf + Uf = f - Vf.
\end{align*}
Using the triangle inequality for $T$-conditional norms and Jensen's inequality,
\begin{align*}
\Vert f - Vf \Vert_{T,p} = \Vert g - Vg \Vert_{T,p} \leq \Vert g \Vert_{T,p} + \Vert Vg \Vert_{T,p} \leq 2\Vert g \Vert_{T,p} = 2\Vert f - Uf \Vert_{T,p}. \tag*{\qed}
\end{align*}
\end{proof}

\begin{thm}
\label{NED shift}
If $(f_{n})_{n \in \mathbb{Z}}$ is near-epoch dependent in $\mathcal{L}^{p}(T)$, $p \in \{1,2,\infty\}$, on $(T_{i}^{j})$, then so is $(f_{n+m})_{n \in \mathbb{Z}}$, $m \in \mathbb{N}$ .
\end{thm}

\begin{proof}
By the properties of $(T_{i}^{j})$, we have that $\mathcal{R}(T_{n+m-k}^{n+m+k}) \subset \mathcal{R}(T_{n-k-m}^{n+k+m})$ for fixed $k,m \in \mathbb{N}$, and so applying the preceding lemma, we get
\begin{align*}
\Vert f_{n+m} - T_{n-k-m}^{n+k+m}f_{n+m} \Vert_{T,p} &\leq 2\Vert f_{n+m} - T_{n+m-k}^{n+m+k}f_{n+m} \Vert_{T,p} \\
&\leq 2 d_{n+m} \xi_{k} \\
&= d_{n}^{\, \prime} \xi_{k},
\end{align*}
where $d_{n}^{\, \prime} = 2d_{n+m} \in \mathcal{L}^{p}(T)_{+}$ and $(\xi_{k})_{k \in \mathbb{N}} \subset \mathcal{R}(T)_{+}$ satisfies $\xi_{k} \rightarrow 0$ in order as $k \rightarrow \infty$. Therefore, we can write
\begin{align*}
\Vert f_{n+m} - T_{n-k}^{n+k}f_{n+m} \Vert_{T,p} \leq d_{n}^{\, \prime} \xi_{k}^{\, \prime},
\end{align*}
where
\begin{align*}
\xi_{k}^{\, \prime} = \left\{\begin{array}{lcl} \xi_{1} & \mbox{if} & k \leq m, \\ \xi_{k-m} & \mbox{if} & k > m. \end{array}\right.
\end{align*}
Therefore, since $\xi_{k}^{\, \prime} \rightarrow 0$ in order as $k \rightarrow \infty$, by definition, we have that $(f_{n+m})_{n \in \mathbb{Z}}$ is near-epoch dependent in $\mathcal{L}^{p}(T)$ on $(T_{i}^{j})$. \qed
\end{proof}

The following corollary arises by combining the preceding theorem with Theorems \ref{NED product 1 infinity and infinity} and \ref{NED product 2 and 2}.

\begin{cor}
\label{NED shift and product 1 and infinity}
Let $(f_{n})_{n \in \mathbb{Z}}$ and $(g_{n})_{n \in \mathbb{Z}}$ be near-epoch dependent in $\mathcal{L}^{p}(T)$ and $\mathcal{L}^{q}(T)$, respectively, on $(T_{i}^{j})$. Then $(f_{n+m}g_{n})_{n \in \mathbb{Z}}$ and $(f_{n}g_{n+m})_{n \in \mathbb{Z}}$, $m \in \mathbb{N}$, are near-epoch dependent in $\mathcal{L}^{r}(T)$ on $(T_{i}^{j})$, where $(p,q,r) \in \{(1,\infty,1), (2,2,1), (\infty, \infty, \infty)\}$.
\end{cor}

We can now prove, under certain conditions, that mixing processes induce mixingales through near-epoch dependence.

\begin{thm}
\label{NED on mixing is mixingale}
Let $(f_{n})_{n \in \mathbb{Z}} \subset \mathcal{L}^{\infty}(T)$ be near-epoch dependent in $\mathcal{L}^{p}(T)$, $p \in \{1, 2, \infty\}$, on $(T_{i}^{j})$. If $(T_{i}^{j})$ is $\alpha_{T}$-mixing or $\varphi_{T}$-mixing and $(f_{n})_{n \in \mathbb{Z}}$ has $T$-conditional mean zero, then the double sequence $(f_{n}, T_{-\infty}^{n})_{n \in \mathbb{Z}}$ is a mixingale in $\mathcal{L}^{1}(T)$.
\end{thm}

\begin{proof}
For fixed $m \in \mathbb{N}$, let $k = \lfloor \frac{m}{2} \rfloor$ be the largest integer not exceeding $\frac{m}{2}$. To verify Definition \ref{mixingale} (i), we note
\begin{align*}
\Vert T_{-\infty}^{n-m} f_{n}\Vert_{T,1} &= \Vert T_{-\infty}^{n-m} ( f_{n} - T_{n-k}^{n+k}f_{n} + T_{n-k}^{n+k}f_{n}) \Vert_{T,1} \\
&\leq \Vert T_{-\infty}^{n-m}(f_{n} - T_{n-k}^{n+k} f_{n})\Vert_{T,1} + \Vert T_{-\infty}^{n-m} T_{n-k}^{n+k} f_{n} \Vert_{T,1}.
\end{align*}
For the first term, using Jensen's inequality and Lyapunov's inequality, respectively,
\begin{align*}
\Vert T_{-\infty}^{n-m}(f_{n} - T_{n-k}^{n+k} f_{n})\Vert_{T,1} &\leq \Vert f_{n} - T_{n-k}^{n+k} f_{n} \Vert_{T,1} \\
&\leq \Vert f_{n} - T_{n-k}^{n+k}f_{n} \Vert_{T,p} \\
&\leq d_{n} \xi_{k},
\end{align*}
where $(d_{n})_{n \in \mathbb{Z}}$ and $(\xi_{m})_{m \in \mathbb{N}}$ are defined as in Definition \ref{NED}. For the second term, note that $T_{n-k}^{n+k} f_{n} \in \mathcal{R}(T_{n-k}^{n+k}) \subset \mathcal{R}(T_{n-k}^{\infty})$. Now, for all $n \in \mathbb{Z}$, since $f_{n} \in \mathcal{L}^{\infty}(T)$, there exists $g_{n} \in \mathcal{R}(T)_{+}$ such that $|f_{n}| \leq g_{n}$, and so
\begin{align*}
|T_{n-k}^{n+k} f_{n}| \leq T_{n-k}^{n+k} |f_{n}|  \leq T_{n-k}^{n+k} g_{n} = T_{n-k}^{n+k} T g_{n} = Tg_{n} = g_{n},
\end{align*}
giving that $T_{n-k}^{n+k} f_{n} \in \mathcal{L}^{\infty}(T)$. Therefore, since $TT_{n-k}^{n+k} f_{n} = 0$, which follows from the compatibility of $(T_{i}^{j})$ with $T$ and the supposition that $Tf_{n} = 0$, we can apply Theorems \ref{alpha mixing inequality} and \ref{phi mixing inequality} to obtain
\begin{align*}
\Vert T_{-\infty}^{n-m} T_{n-k}^{n+k} f_{n} \Vert_{T,1} &= \Vert T_{-\infty}^{n-m} T_{n-k}^{n+k} f_{n} - T T_{n-k}^{n+k} f_{n} \Vert_{T,1} \\
&\leq 2 \min\{2 \alpha_{T}(T_{-\infty}^{n-m}, T_{n-k}^{\infty}), \varphi_{T}(T_{-\infty}^{n-m}, T_{n-k}^{\infty})\} \Vert f_{n} \Vert_{T,\infty} \\
&= 2 \min\{2 \alpha_{T}(T_{-\infty}^{n-m}, T_{n-m+k}^{\infty}), \varphi_{T}(T_{-\infty}^{n-m}, T_{n-m+k}^{\infty})\} \Vert f_{n} \Vert_{T,\infty} \\
&\leq 2 \min\{2 \alpha_{T,k}, \varphi_{T,k}\} \Vert f_{n} \Vert_{T,\infty}.
\end{align*}
Combining the above results gives
\begin{align*}
\Vert T_{-\infty}^{n-m}f_{n} \Vert_{T,1} &\leq d_{n} \xi_{k} + 2 \min\{2\alpha_{T,k} , \varphi_{T,k} \} \Vert f_{n} \Vert_{T,\infty} \\
&\leq c_{n} \phi_{m},
\end{align*}
where $c_{n} = d_{n} \vee \Vert f_{n} \Vert_{T,\infty} \in \mathcal{L}^{1}(T)_{+}$ and $\phi_{m} = 2(\xi_{k} + \min\{2 \alpha_{T,k}, \varphi_{T,k}\}) \in \mathcal{R}(T)_{+}$, for all $n \in \mathbb{Z}$ and $m \in \mathbb{N}$. Note that by supposition, $\phi_{m} \rightarrow 0$ in order as $m \rightarrow \infty$. For Definition \ref{mixingale} (ii), since $\mathcal{R}(T_{n-k-1}^{n+k+1}) \subset \mathcal{R}(T_{n-m}^{n+m}) \subset \mathcal{R}(T_{-\infty}^{n+m})$, we have, using Lemma \ref{NED inequality lemma}, that
\begin{align*}
\Vert f_{n} - T_{-\infty}^{n+m}f_{n} \Vert_{T,1} &\leq 2 \Vert f_{n} - T_{n-k-1}^{n+k+1}f_{n} \Vert_{T,1} \\
&\leq 2 \Vert f_{n} - T_{n-k-1}^{n+k+1} f_{n} \Vert_{T,p} \\
&\leq 2 d_{n} \xi_{k+1} \\
&\leq c_{n} 2(\xi_{k+1} + \min\{2\alpha_{T,k+1}, \varphi_{T,k+1}\}) \\
&= c_{n} \phi_{m+2} \\
&\leq c_{n} \phi_{m+1},
\end{align*}
where the final inequality follows since it can be assumed, without loss of generality, that $(\phi_{m})_{m \in \mathbb{N}}$ is a decreasing sequence. This concludes the proof. \qed
\end{proof}

The preceding theorem carries significant implications for the study of near-epoch dependence. This is the case since it is now possible, by appealing to the theory developed for mixingales, to establish important results for near-epoch dependent sequences that are otherwise inaccessible. As an example, combining Theorems \ref{mixingale weak law of large numbers} and \ref{NED on mixing is mixingale} gives the following corollary.

\begin{cor}
\label{NED weak law of large numbers}
Let $(f_{n})_{n \in \mathbb{Z}} \subset \mathcal{L}^{\infty}(T)$ be near-epoch dependent in $\mathcal{L}^{p}(T)$, $p \in \{1,2,\infty\}$, on $(T_{i}^{j})$ with $(d_{n})_{n \in \mathbb{Z}}$ as defined in Definition \ref{NED}, and where $|f_{n}| \leq g_{n} \in \mathcal{R}(T)_{+}$ for all $n \in \mathbb{Z}$. Furthermore, suppose that $(f_{n})_{n \in \mathbb{Z}}$ is $T$-uniform and has $T$-conditional mean zero, and that $(T_{i}^{j})$ is $\alpha_{T}$-mixing or $\varphi_{T}$-mixing.
\begin{enumerate}
\item[(i)] If $\displaystyle{\left(\frac{1}{m} \sum_{i=n+1}^{n+m} d_{i} \vee g_{i} \right)_{m \in \mathbb{N}}}$ is bounded in $E$ uniformly in $n \in \mathbb{Z}$, then
\begin{align*}
T|\overline{f}_{n,m}| = T\left|\frac{1}{m} \sum_{i=n+1}^{n+m} f_{i} \right| \rightarrow 0 \text{ in order as } m \rightarrow \infty, \text{ uniformly in } n \in \mathbb{Z}.
\end{align*}
\item[(ii)] If $d_{n} \vee g_{n} = T|f_{n}|$ for all $n \in \mathbb{Z}$, then
\begin{align*}
T|\overline{f}_{n,m}| = T\left|\frac{1}{m} \sum_{i=n+1}^{n+m} f_{i} \right| \rightarrow 0 \text{ in order as } m \rightarrow \infty, \text{ uniformly in } n \in \mathbb{Z}.
\end{align*}
\end{enumerate}
\end{cor}


\newsection{Application to autoregressives process of order $\mathbf{1}$}

To demonstrate the workings of Definition \ref{NED}, we will consider a non-trivial example of a near-epoch dependent process, that of an autoregressive process of order 1 in $\mathcal{L}^{2}(T)$. To enable such a treatment, we require several preliminary results, the first of which is proved in \cite{bern}.

\begin{lem}
\label{Sf squared in R(S)}
Let $E$ be a $T$-universally complete Riesz space with weak order unit $e = Te$, where $T$ is a strictly positive conditional expectation operator on $E$, and let $U$ and $V$ be conditional expectation operators on $E$ compatible with $T$. If $g,h \in \mathcal{L}^{2}(T)$ and $VU = UV = V$, then $U(g \cdot Vh) = Vh \cdot Ug$.
\end{lem}

\begin{thm}
\label{mse minimised}
Let $E$ be a $T$-universally complete Riesz space with weak order unit $e=Te$, where $T$ is a strictly positive conditional expectation operator on $E$, and let $S$ be a conditional expectation operator on $E$ compatible with $T$. For $f \in \mathcal{L}^{2}(T)$, $\left\Vert f-g \right\Vert_{T,2}$ is minimised over $g \in \mathcal{R}(S) \cap \mathcal{L}^{2}(T)$ by $g = Sf$.
\end{thm}

\begin{proof}
Recall that $Sf \in \mathcal{L}^{2}(T)$. Then for $g \in \mathcal{R}(S) \cap \mathcal{L}^{2}(T)$,
\begin{align}
\label{mse min}
\left\Vert f - g \right\Vert_{T,2}^{2} &= T|f-g|^{2} \nonumber\\
&= T(f-g)^{2} \nonumber \\
&= T(f - Sf + Sf - g)^{2} \nonumber \\
&= T(f - Sf)^{2} + 2T[(f-Sf)(Sf-g)] + T(Sf-g)^{2}.
\end{align}
Squaring out the middle term, we have
\begin{align*}
T[(f-Sf)(Sf-g)] &= T\big(f \cdot Sf - fg - (Sf)^{2} + g \cdot Sf\big) \\
&= T(f \cdot Sf) - T(Sf)^{2} + T(g \cdot Sf) -T(fg).
\end{align*}
Using Lemma \ref{Sf squared in R(S)} with $U=V=S$, we have
\begin{align*}
T(f \cdot Sf) = TS(f \cdot Sf) = T(Sf \cdot Sf) = T(Sf)^{2}.
\end{align*}
Since $g \in \mathcal{R}(S)$, we have
\begin{align*}
T(fg) = T(f \cdot Sg) = TS(f \cdot Sg) = T(Sg \cdot Sf) = T(g \cdot Sf).
\end{align*}
Therefore, the middle term of (\ref{mse min}) is zero, which gives
\begin{align*}
\left\Vert f-g \right\Vert_{T,2}^{2} = T(f - Sf)^{2} + T(Sf - g)^{2}.
\end{align*}
Since the first term above is independent of $g$ and the second term is necessarily non-negative, it follows that the minimum is attained by setting $g = Sf$. \qed
\end{proof}

\begin{lem}
\label{infinite sum in R(T)}
Let $E$ be a $T$-universally complete Riesz space with weak order unit $e = Te$, where $T$ is a strictly positive conditional expectation operator on $E$. Let $f \in \mathcal{R}(T)_{+}$ such that $P_{(e-f)^{+}} = I$, then $\sum_{i=0}^{\infty} f^{i}$ converges in order in $\mathcal{R}(T)$ to $\frac{e}{e-f}$.
\end{lem}

\begin{proof}
If convergent, we can multiply the sum through by $e-f$ to obtain
\begin{align*}
(e-f) \sum_{i=0}^{\infty} f^{i} &= \sum_{i=0}^{\infty} f^{i} - \sum_{i=0}^{\infty} f^{i+1} 
= \sum_{i=0}^{\infty} f^{i} - \sum_{i=1}^{\infty} f^{i} 
= f^{0} 
= e.
\end{align*}
To prove convergence, let
\begin{align*}
Q_{n} = P_{\big((1-\frac{1}{2^{n}})e - f\big)^{+}} \Big( I - P_{\big((1 - \frac{1}{2^{n-1}})e - f\big)^{+}}\Big).
\end{align*}
Then $Q_{n}Q_{m} = 0$ for all $n \neq m$, and as $P_{(e-f)^{+}} = I$ we have $\vee_{n=1}^{\infty} Q_{n} = I$. As $\mathcal{R}(T)$ is universally complete and $Q_{n} \wedge Q_{m} = 0$ for all $n \neq m$, we can define
\begin{align*}
g = \bigvee_{n=1}^{\infty} 2^{n} Q_{n} e = \sum_{n=1}^{\infty} 2^{n} Q_{n} e \in \mathcal{R}(T)_{+}.
\end{align*}
Now
\begin{align*}
Q_{m} \sum_{i=0}^{n} f^{i} &= \sum_{i=0}^{n} (Q_{m} f)^{i} 
\leq \sum_{i=0}^{n} \Big(1 - \frac{1}{2^{m}}\Big)^{i} Q_{m} e 
\leq 2^{m} Q_{m} e 
\leq Q_{m} g.
\end{align*}
Taking suprema over $Q_{m}$, we have, for all $n \in \mathbb{N}$,
\begin{align*}
\sum_{i=0}^{n} f^{i} \leq g.
\end{align*}
As $\mathcal{R}(T)$ is universally complete, it is certainly Dedekind complete, and as $(\sum_{i=0}^{n} f^{i})_{n \in \mathbb{N}}$ is an increasing sequence in $\mathcal{R}(T)$ and is bounded above by $g$, we have that $\sum_{i=0}^{\infty} f^{i}$ converges in $\mathcal{R}(T)$. \qed
\end{proof}

\begin{lem}
\label{zero limit in R(T)}
Let $E$ be a $T$-universally complete Riesz space with weak order unit $e = Te$, where $T$ is a strictly positive conditional expectation operator on $E$. Let $f \in \mathcal{R}(T)_{+}$ such that $P_{(e-f)^{+}} = I$, then $f^{m} \rightarrow 0$ in order as $m \rightarrow \infty$.
\end{lem}

\begin{proof}
Let $h = \inf\{f^{m} : m \in \mathbb{N}\} \in \mathcal{R}(T)_{+}$. Since $f^{m} \downarrow h$, if $h = 0$, then we have the result. Therefore suppose that $h \neq 0$, then since $\mathcal{L}^{1}(T)$ is Archimedean, there exists $n \in \mathbb{N}$ such that $nh > e$, giving $h \nleq \frac{1}{n} e$, which is to say that
\begin{align*}
Q = P_{(h - \frac{1}{n}e)^{+}} \neq 0.
\end{align*}
Therefore $\inf\{Qf^{m} : m \in \mathbb{N}\} = Qh \geq \frac{1}{n} Qe$, giving $Qf^{m} = (Qf)^{m} \geq \frac{1}{n} Qe$ for all $m \in \mathbb{N}$. Therefore
\begin{align*}
Qf \geq \frac{1}{n^{\frac{1}{m}}} Qe
\end{align*}
for all $m \in \mathbb{N}$. In particular, taking $m \rightarrow \infty$ gives $Q(e-f) \leq 0$, which implies that $Q \leq I - P_{(e-f)^{+}} = 0$. This contradicts $Q \neq 0$, and so we have that $h = 0$. \qed
\end{proof}

We are now in a position to analyze the autoregressive process of order 1 in the context of near-epoch dependence. For a similar exposition in the classical setting, see \cite{gallantwhite}.

\begin{defs}
\label{AR1}
Let $E$ be a $T$-universally complete Riesz space with weak order unit $e = Te$, where $T$ is a strictly positive conditional expectation operator on $E$. The sequence $(f_{n})_{n \in \mathbb{N}} \subset \mathcal{L}^{2}(T)$ is said to be a $T$-conditional autoregressive process of order 1 if, for all $n \in \mathbb{N}$,
\begin{align*}
f_{n} = \theta f_{n-1} + \varepsilon_{n},
\end{align*}
where $f_{0} = 0$, $\theta \in \mathcal{R}(T)$, and the sequence $(\varepsilon_{n})_{n \in \mathbb{N}} \subset \mathcal{L}^{2}(T)$ has $T$-conditional mean zero.
\end{defs}

Note that the sequences $(f_{n})_{n \in \mathbb{N}}$ and $(\varepsilon_{n})_{n \in \mathbb{N}}$ given in Definition \ref{AR1} can be extended arbitrarily to the index set $\mathbb{Z}$ by setting $f_{n} = \varepsilon_{n} = 0$ for all $n \leq 0$. We now show that if $(\Vert \varepsilon_{n} \Vert_{T,2})_{n \in \mathbb{N}}$ is bounded in $E$ by $g \in E_{+}$, say, and $\theta \in \mathcal{R}(T)$ satisfies $P_{(e-|\theta|)^{+}} = I$, the sequence $(f_{n})_{n \in \mathbb{N}}$ in the preceding definition is near-epoch dependent in $\mathcal{L}^{2}(T)$ on $(\varepsilon_{n})_{n \in \mathbb{N}}$, or, more precisely, on the family of conditional expectation operators $(T_{i}^{j})$, where $\mathcal{R}(T_{i}^{j}) = \langle \{ \varepsilon_{r} : i \leq r \leq j \} \cup \mathcal{R}(T) \rangle$, where the existence of the conditional expectation operators $(T_{i}^{j})$, as well as their compatibility with $T$, is assured by \cite{ando-douglas}.

To start, it is easy to show, by induction on $n$ and from Definition \ref{AR1}, that for all $n \in \mathbb{N}$,
\begin{align*}
f_{n} = \sum_{i=0}^{n-1} \theta^{i} \varepsilon_{n-i} = \sum_{i=0}^{\infty} \theta^{i} \varepsilon_{n-i}.
\end{align*}
Therefore, since $\sum_{i=0}^{m} \theta^{i} \varepsilon_{n-i} \in \mathcal{R}(T_{n-m}^{n+m}) \cap \mathcal{L}^{2}(T)$, we have by Theorem \ref{mse minimised},
\begin{align*}
\Vert f_{n} - T_{n-m}^{n+m} f_{n} \Vert_{T,2} \leq \, &\left\Vert f_{n} - \sum_{i=0}^{m} \theta^{i} \varepsilon_{n-i} \right\Vert_{T,2} \\
= \, &\left\Vert \sum_{i=m+1}^{\infty} \theta^{i} \varepsilon_{n-i} \right\Vert_{T,2} \\
= \, &\left\Vert \theta^{m} \sum_{i=1}^{\infty} \theta^{i} \varepsilon_{n-m-i} \right\Vert_{T,2}.
\end{align*}
Then, since the above summation is finite, we can apply homogeneity and the triangle inequality for $\left\Vert \, \cdot \, \right\Vert_{T,2}$ inductively to obtain
\begin{align*}
\Vert f_{n} - T_{n-m}^{n+m} f_{n} \Vert_{T,2} &\leq |\theta|^{m} \sum_{i=1}^{\infty} |\theta|^{i} \Vert \varepsilon_{n-m-i} \Vert_{T,2} \\
&\leq |\theta|^{m} \sum_{i=1}^{\infty} |\theta|^{i} \, g \\
&= |\theta|^{m+1} \bigg(\sum_{i=0}^{\infty} |\theta|^{i} \bigg) g \\
&= \frac{|\theta|^{m+1}}{e-|\theta|} \, g, \text{ by Lemma \ref{infinite sum in R(T)} } \\
&= g \, \xi_{m},
\end{align*}
where $\xi_{m} = \frac{|\theta|^{m+1}}{e - |\theta|} \in \mathcal{R}(T)_{+}$. Since the above holds for all $m, n \in \mathbb{N}$ and since $\xi_{m} \rightarrow 0$ in order as $m \rightarrow \infty$, which follows from Lemma \ref{zero limit in R(T)}, we have, under the conditions set out, that the $T$-conditional autoregressive process of order 1 is near-epoch dependent in $\mathcal{L}^{2}(T)$ on $(T_{i}^{j})$, where $\mathcal{R}(T_{i}^{j}) = \langle \{ \varepsilon_{r} : i \leq r \leq j \} \cup \mathcal{R}(T) \rangle$.


\end{document}